\documentclass[10pt,reqno,a4paper]{amsart}
\usepackage{amsfonts,amsmath,amssymb,times}
\usepackage{algorithmic,algorithm}
\usepackage{amsthm,array}%,a4wide

\usepackage{hyperref}
\usepackage{color}
\definecolor{blau}{rgb}{0,0,0.75} %color for in-document links
\hypersetup{colorlinks,linkcolor=blau,citecolor=blue}

\allowdisplaybreaks

\def\RR{{\mathbb {R}}}
\def\NN{{\mathbb {N}}}

\DeclareMathOperator{\Li}{Li}
\DeclareMathOperator{\stuffi}{stuffle}
\newtheorem{coroll}{Corollary}
\newtheorem{theorem}{Theorem}
\newtheorem{lemma}{Lemma}

\theoremstyle{definition}
\newtheorem{remark}{Remark}

\title{On functions of Arakawa and Kaneko and multiple zeta functions}

\author[M.~Kuba]{Markus Kuba}

\address{Markus Kuba\\
Institut f{\"u}r Diskrete Mathematik und Geometrie\\
Technische Universit\"at Wien\\
Wiedner Hauptstr. 8-10/104\\
1040 Wien, Austria} %
\email{kuba@dmg.tuwien.ac.at}

\begin{document}

\begin{abstract}
We study for $s\in\NN$ the functions $\xi_{k}(s)=\frac{1}{\Gamma(s)}\int_{0}^{\infty}\frac{t^{s-1}}{e^t-1}\Li_{k}(1-e^{-t})dt$, and more generally $\xi_{k_1,\dots,k_r}(s)=\frac{1}{\Gamma(s)}\int_{0}^{\infty}\frac{t^{s-1}}{e^t-1}\Li_{k_1,\dots,k_r}(1-e^{-t})dt$, introduced by Arakawa and Kaneko~\cite{Arakawa} and relate them with (finite) multiple zeta functions, partially answering a question of~\cite{Arakawa}. In particular, we give an alternative proof of a result of Ohno~\cite{Ohno2}.
\end{abstract}


\maketitle

\section{Introduction}
\noindent Let $\Li_{k_1,\dots,k_r}(z)$ denote the multiple polylogarithm function %with parameters $k_1,\dots,k_r$,
defined by
\begin{equation*}
\Li_{k_1,\dots,k_r}(z)=\sum_{n_1>n_2>\dots>n_r\ge 1}\frac{z^{n_1}}{n_1^{k_1}n_2^{k_2}\dots n_{r}^{k_r}},
\end{equation*}
with $k_1\in\NN\setminus\{1\}$ and $k_i\in\NN=\{1,2,\dots\}$, $2\le i\le r$, and $|z|\le 1$.
For $z=1$ the multiple polylogarithm function $\Li_{k_1,\dots,k_r}(1)=\zeta(k_1,\dots,k_r)$ simplifies to a multiple zeta function, also called multiple zeta value, where $\zeta(k_1,\dots,k_r)$ and $\zeta_N(k_1,\dots,k_r)$ denote the (finite) multiple zeta function defined by 
\begin{equation*}
\begin{split}
    \label{zeta}
      \zeta(k_1,\dots,k_{r}) &= \sum_{  n_1 > n_2 > \dots >  n_{r}\ge 1}\frac{1}{n_1^{k_1} n_2^{k_2} \dots n_{r}^{k_{r}}},\\
  \zeta_{N}(k_1,\dots,k_{r}) &= \sum_{ N\ge n_1 > n_2 > \dots >
    n_{\ell} \ge 1}\frac{1}{n_1^{ka_1} n_2^{k_2} \dots n_{r}^{k_{r}}},
\end{split}
\end{equation*}
with $k_1\in\NN\setminus\{1\}$, and $k_2,\dots,k_r\in\NN$ for the infinite series and $N,k_1,\dots,k_r\in\NN$ for the finite counterpart.
Arakawa and Kaneko~\cite{Arakawa} introduced the function $\xi_{k}(s)$, and the more general function $\xi_{k_1,\dots,k_r}(s)$,
defined by
\begin{equation*}
\begin{split}
\xi_{k}(s)&=\frac{1}{\Gamma(s)}\int_{0}^{\infty}\frac{t^{s-1}}{e^t-1}\Li_{k}(1-e^{-t})dt,\\
\xi_{k_1,\dots,k_r}(s)&=\frac{1}{\Gamma(s)}\int_{0}^{\infty}\frac{t^{s-1}}{e^t-1}\Li_{k_1,\dots,k_r}(1-e^{-t})dt,
\end{split}
\end{equation*}
respectively, being absolut convergent for $\Re(s)>0$, and related them for special choices of $s$ and $k_1,\dots k_r$ to multiple zeta functions.
Ohno~\cite{Ohno2} obtained a result for $\xi_{k}(n)$, with $n\in\NN$, using his generalization of the duality and sum formulas for
multiple zeta functions.

\smallskip

We will provide for $n\in\NN$ evaluations of the function $\xi_{k_1,\dots,k_r}(n)$ to multiple zeta functions, partially answering a question of~\cite{Arakawa}; in particular we give a short and simple proof of Ohno's result~\cite{Ohno2}. For the evaluation of the general case $\xi_{k_1,\dots,k_r}(n)$ we use a finite version of the so-called stuffle identity for multiple zeta functions.
Subsequently, we will utilise a variant of (finite) multiple zeta functions, called the multiple zeta star functions
or non-strict multiple zeta functions $\zeta_{N}^{\ast}(k_1,\dots,k_{r})$, which recently attracted some interest,\cite{Ohno,Ohno2,Ohno3,Ohno4,Hoffman,Mune,Eie,Zudi}
where the summation indices satisfy $N\ge n_1\ge n_2\ge \dots \ge n_{r} \ge 1$ in contrast to $N\ge n_1> n_2> \dots > n_{r}> 1$, as in the usual definition~\eqref{zeta},
\begin{equation*}
\zeta_{N}^{\ast}(k_1,\dots,k_{r})=\sum_{ N\ge n_1 \ge n_2\ge \dots \ge
    n_{r} \ge 1}\frac{1}{n_1^{k_1} n_2^{k_2} \dots n_r^{k_{r}}},
\end{equation*}
with $N,k_1,\dots,k_r\in\NN$. The star form can be converted into ordinary finite multiple zeta functions by considering all possible deletions of commas, i.e. 
\begin{equation}
\label{OHNOconvert}
\zeta_{N}^{\ast}(k_1,\dots,k_{r})=
\sum_{h=1}^{r}\sum_{1\le \ell_1<\ell_2<\dots<\ell_{h-1}<r}\zeta_N\biggl(\sum_{i_1=1}^{\ell_1}k_{i_1},\sum_{i_2=\ell_1+1}^{\ell_2}k_{i_2},\dots,\sum_{i_h=\ell_{h-1}+1}^{r}k_{i_h}\biggr);
\end{equation}
note that the first term $h=1$ should be interpreted as $\zeta_N(\sum_{i_1=\ell_{0}+1}^{r}k_{i_1})$, subject to $\ell_0=0$. The notation $\zeta_{N}^{\ast}(k_1,\dots,k_r)$ is chosen in analogy with Aoki and Ohno~\cite{Ohno}
where infinite counterparts of $\zeta_{N}^{\ast}(k_1,\dots,k_r)$ have been treated. 
First we will study the instructive case of $\xi_k(n)$, reproving the result of Ohno. Then we will state our main result
concerning the evaluation of $\xi_{k_1,\dots,k_r}(n)$ into multiple zeta functions.

\section{A simple evaluation of $\xi_k(n)$}
Ohno~\cite{Ohno2} evaluated the sum $\xi_k(n)$ for $k,n\in\NN$ applying his generalization of the duality and sum formulas for
multiple zeta functions to a result of Arakawa and Kaneko~\cite{Arakawa}. In the following we will give an alternative simple and self-contained derivation of his result, stated in Theorem~\ref{OHNOthe1}.
In order to evaluate $\xi_k(n)$ for $k,n\in\NN$ we only use the two basic facts stated below.
\begin{equation}
\begin{split}
\label{OHNObasic}
&\frac{1}{\Gamma(n)}\int_{0}^{\infty}t^{k-1}e^{-tn}dt=\frac{1}{n^k},\quad\text{for}\quad n\in\NN,\\ &\sum_{\ell=1}^{n}\binom{n}{\ell}\frac{(-1)^{\ell-1}}{\ell^r}=\zeta_{n}^{\ast}(\underbrace{1,\dots,1}_{r})=\zeta_{n}^{\ast}(\{1\}_r).
\end{split}
\end{equation}
The second identity can be immediately deduced by repeated usage of the formula $\binom{n}{k}=\sum_{\ell=k}^{n}\binom{\ell-1}{k-1}$. We proceed as follows.
\begin{equation*}
\xi_{k}(n)=\frac{1}{\Gamma(n)}\int_{0}^{\infty}\frac{t^{n-1}}{e^t-1}\Li_{k}(1-e^{-t})dt
=\frac{1}{\Gamma(n)}\int_{0}^{\infty}t^{n-1} e^{-t}\sum_{m \ge 1}\frac{(1-e^{-t})^{m-1}}{m^k}dt.
\end{equation*}
We expand $(1-e^{-t})^{m-1}$ by the binomial theorem and interchange summation and integration.
According to~\eqref{OHNObasic} we obtain
\begin{equation*}
\begin{split}
\xi_{k}(n)&=\sum_{m \ge 1}\frac{1}{m^k}\sum_{\ell=0}^{m-1}\binom{m-1}{\ell}\frac{(-1)^{\ell}}{\Gamma(n)}\int_{0}^{\infty}t^{n-1} e^{-(\ell+1)t}dt\\
&=\sum_{m \ge 1}\frac{1}{m^k}\sum_{\ell=0}^{m-1}\binom{m-1}{\ell}\frac{(-1)^{\ell}}{(\ell+1)^n}.
\end{split}
\end{equation*}
Since $\binom{m-1}{\ell}=\binom{m}{\ell+1}\frac{\ell+1}{m}$, we get according to~\eqref{OHNObasic} after an index shift the following result
\begin{theorem}[Ohno~\cite{Ohno2}]
\label{OHNOthe1}
For $k,n\in\NN$ the function $\xi_k(n)$ is given by
\begin{equation*}
\xi_k(n)=\sum_{m_1\ge m_2\ge \dots \ge m_{n}\ge 1}\frac{1}{m_1^{k+1}m_2\dots m_n}=
\zeta^{\ast}(k+1,\{1\}_{n-1}).
\end{equation*}
\end{theorem}
Note that one can convert the multiple star zeta function above into ordinary multiple zeta functions according to~\eqref{OHNOconvert} (with respect to the corresponding relation for infinite series), or can directly simplify the multiple zeta star function using (cycle) sum formulas, see i.e.~Ohno and Wakabayashi~\cite{Ohno3} or Ohno and Okuda~\cite{Ohno4}.

\section{General case}
In the general case of $\xi_{k_1,\dots,k_r}(n)$ we will prove the following result.
\begin{theorem}
\label{OHNOthe2}
For $k_1,\dots,k_r,n\in\NN$ the function $\xi_{k_1,\dots,k_r}(n)$ is given by
\begin{equation*}
\xi_{k_1,\dots,k_r}(n)=\sum_{n_1\ge 1}\frac{\zeta_{n_1}^{\ast}(\{1\}_{n-1})\zeta_{n_1-1}(k_2,\dots,k_r)}{n_1^{k_1+1}}.
\end{equation*}
Furthermore, $\xi_{k_1,\dots,k_r}(n)$ can be evaluated into sums of multiple zeta functions.
\end{theorem}
The explicit evaluation of $\xi_{k_1,\dots,k_r}(n)$ will be given in Corollary~\ref{OHNOthe3}. We have
\begin{equation*}
\begin{split}
\xi_{k_1,\dots,k_r}(n)&=\frac{1}{\Gamma(n)}\int_{0}^{\infty}\frac{t^{n-1}}{e^t-1}\Li_{k_1,\dots,k_r}(1-e^{-t})dt\\
&=\frac{1}{\Gamma(n)}\int_{0}^{\infty}t^{n-1}e^{-t}
\sum_{n_1>n_2>\dots>n_r\ge 1}\frac{(1-e^{-t})^{n_1-1}}{n_1^{k_1}n_2^{k_2}\dots n_{r}^{k_r}}dt.
\end{split}
\end{equation*}
Proceeding as before we expand $(1-e^{-t})^{n_1-1}$ by the binomial theorem and interchange summation and integration. We get
\begin{equation*}
\begin{split}
\xi_{k_1,\dots,k_r}(n)&=
\sum_{n_1>n_2>\dots>n_r\ge 1}\frac{1}{n_1^{k_1}n_2^{k_2}\dots n_{r}^{k_r}}
\sum_{\ell=0}^{n_1-1}\binom{n_1-1}{\ell}\frac{(-1)^{\ell}}{\Gamma(n)}\int_{0}^{\infty}t^{n-1}e^{-(\ell+1)t}dt.
\end{split}
\end{equation*}
According to~\eqref{OHNObasic} we obtain
\begin{equation*}
\xi_{k_1,\dots,k_r}(n)=\sum_{n_1>n_2>\dots>n_r\ge 1}\frac{\zeta_{n_1}^{\ast}(\{1\}_{n-1})}{n_1^{k_1+1}n_2^{k_2}\dots n_{r}^{k_r}}
=\sum_{n_1\ge 1}\frac{\zeta_{n_1}^{\ast}(\{1\}_{n-1})\zeta_{n_1-1}(k_2,\dots,k_r)}{n_1^{k_1+1}}.
\end{equation*}
This proves the first part of Theorem~\ref{OHNOthe2}.

\smallskip

Concerning the second part we proceed as follows. We will evaluate the
product $S$ of finite multiple zeta (star) functions
\begin{equation*}
S=S_{n_1}(n,k_2,\dots,k_2)=\zeta_{n_1}^{\ast}(\{1\}_{n-1})\zeta_{n_1-1}(k_2,\dots,k_r)
\end{equation*}
into sums of finite multiple zeta functions $\zeta_{n_1-1}(\mathbf{f})$, for some $\mathbf{f}=(f_1,\dots,f_j)$, with $f_i\in\NN$, $1\le i\le j$, which will prove the second part of the stated result. By~\eqref{OHNOconvert} we can write $\zeta_{n_1}^{\ast}(\{1\}_{n-1})$ in terms of ordinary finite multiple zeta functions
\begin{equation*}
\zeta_{n_1}^{\ast}(\{1\}_{n-1})=
\sum_{h=1}^{n-1}\sum_{1\le \ell_1<\ell_2<\dots<\ell_{h-1}<n-1}\zeta_{n_1}\bigl(\ell_1,\ell_2-\ell_1,\dots,n-\ell_{h-1}-1\bigr);
\end{equation*}
for example $\zeta_{n_1}^{\ast}(\{1\}_3)=\zeta_{n_1}(3)+\zeta_{n_1}(1,2)+\zeta_{n_1}(2,1)+\zeta_{n_1}(1,1,1)$.
We can convert finite zeta functions $\zeta_{N}(a_1,\dots,a_r)$ into finite zeta functions $\zeta_{N-1}(b_1,\dots,b_s)$ by
\begin{equation*}
\zeta_{N}(a_1,\dots,a_r)=\zeta_{N-1}(a_1,\dots,a_r) + \frac{1}{N^{a_1}}\zeta_{N-1}(a_2,\dots,a_r).
\end{equation*}
Consequently, we can express the product $S=\zeta_{n_1}^{\ast}(\{1\}_{n-1})\zeta_{n_1-1}(k_2,\dots,k_r)$ of finite multiple zeta (star) functions in the following way.
\begin{equation}
\begin{split}
\label{OHNOs}
S&=
\sum_{h=1}^{n-1}\sum_{1\le \ell_1<\ell_2<\dots<\ell_{h-1}<n-1}
\zeta_{n_1-1}\bigl(\ell_1,\ell_2-\ell_1,\dots,n-\ell_{h-1}-1\bigr)\zeta_{n_1-1}(k_2,\dots,k_r)\\
&\quad+ \sum_{h=1}^{n-1}\sum_{1\le \ell_1<\ell_2<\dots<\ell_{h-1}<n-1}\frac{\zeta_{n_1-1}\bigl(\ell_2,\dots,n-\ell_{h-1}-1\bigr)\zeta_{n_1-1}(k_2,\dots,k_r)}{n_1^{\ell_1}}\Big).
\end{split}
\end{equation}
Now we use finite versions of so-called \emph{stuffle identities}, see i.e.~Borwein et al.~\cite{Bor}. Stuffle identities provide evaluation of products of multiple zeta functions $\zeta(\mathbf{k})\zeta(\mathbf{h})$ into into sums of multiple zeta functions $\zeta(\mathbf{k})\zeta(\mathbf{h})=\sum_{\mathbf{f}\in\stuffi(\mathbf{k},\mathbf{h})}\zeta(\mathbf{f})$; here $\zeta(\mathbf{k})=\zeta(k_1,\dots,k_r)$ and $\zeta(\mathbf{h})=\zeta(h_1,\dots,h_s)$. For our problem we need finite versions of the stuffle identities, providing evaluations of products of finite multiple zeta functions $\zeta_N(\mathbf{k})\zeta_N(\mathbf{h})$ into into sums of finite multiple zeta functions $\zeta_N(\mathbf{k})\zeta_N(\mathbf{h})=\sum_{\mathbf{f}\in\stuffi(\mathbf{k},\mathbf{h})}\zeta_N(\mathbf{f})$.
This would lead then to an evaluation of $\xi_{k_1,\dots,k_r}(n)$ into single finite multiple zeta functions. 

\smallskip

Following~\cite{Bor} we define for two given strings $\mathbf{k}=(k_1,\dots,k_r)$ and $\mathbf{h}=(h_1,\dots,h_s)$
the set $\stuffi(\mathbf{k},\mathbf{h})$ as the smallest set of strings over the alphabet $\mathcal{A}$, defined by  $$\mathcal{A}=\{k_1,\dots,k_r,h_1,\dots,h_s,``+'',``,'',``('',``)''\}$$
satisfying $(\mathbf{k},\mathbf{h})=(k_1,\dots,k_r,h_1,\dots,h_s)\in\stuffi(\mathbf{k},\mathbf{h})$
and further if a string of the form $(U,k_n,h_m,V)\in\stuffi(\mathbf{k},\mathbf{h})$, then so are the strings $(U,h_m,k_n,V)\in\stuffi(\mathbf{k},\mathbf{h})$ and $(U,k_n+h_m,V)\in\stuffi(\mathbf{k},\mathbf{h})$. Stuffle identities arise from the definition of (finite) multiple zeta functions in terms of sums; the term \emph{stuffle} derives from the manner in which the two upper strings are
combined. Other closely related identities are due to different representations of multiple zeta functions (see for example~\cite{Bor}).
We will use the following result
\begin{lemma}[Stuffle identity; finite version]
\label{OHNOstufflelem}
Let $\zeta_N(\mathbf{k})=\zeta_N(k_1,\dots,k_r)$ and $\zeta_N(\mathbf{h})=\zeta_N(h_1,\dots,h_s)$, with $N,r,s\in\NN$ and $k_i,h_j\in\NN$, $1\le i\le r$, $1\le j\le s$. Then,
\begin{equation}
\label{OHNOstuffle}
\zeta_{N}(\mathbf{k})\zeta_N(\mathbf{h})=\sum_{\mathbf{f}\in\stuffi(\mathbf{k},\mathbf{h})}\zeta_N(\mathbf{f}).
\end{equation}
\end{lemma}
\begin{remark}
Note that the conditions on the parameters $k_i,h_j$ can be relaxed in various ways, i.e.~$k_i,h_j\in\RR^{+}$.
Typical examples of (finite) stuffle identities read as follows, 
\begin{equation*}
\begin{split}
\zeta_{N}(r)\zeta_N(t)&=\zeta_N(r,t)+\zeta_N(t+r)+\zeta_N(t,r),\\
\zeta_{N}(r,s)\zeta_N(t)&= \zeta_{N}(r,s,t)+\zeta_{N}(r,s+t)+\zeta_{N}(r,t,s)+\zeta_{N}(r+t,s)+\zeta_{N}(t,r,s);\\
%\zeta_{N}(r,s)\zeta_N(t,u)&= \zeta_{N}(r,s,t,u)+\zeta_{N}(r,s+t,u)+\zeta_{N}(r+t,s,u)+\zeta_{N}(r+t,s+u)+\zeta_{N}(r+t,u,s)\\
%&\quad+\zeta_{N}(t,r,s,u)+\zeta_{N}(t,r,u,s)+\zeta_{N}(t,r+u,u,s)+\zeta_{N}(t,u,r,s),\\
\end{split}
\end{equation*}
we refer to~\cite{Bor} concerning infinite counterparts, and also to Borwein and Girgensohn~\cite{Bor2} where the second identity is implicitly derived.
Stuffle identities for finite multiple zeta functions seem to be natural, since one does not have to exclude the cases $k_1=1$ or $h_1=1$ in contrast
to infinite multiple zeta functions. 
\end{remark}
\begin{proof}[Proof of Lemma~\ref{OHNOstufflelem} (Sketch)]
In order prove this result in an elementary way, one can use
induction with respect to the total length
$|\mathbf{k}|+|\mathbf{h}|$. We do not want to give a full proof,
since we believe that~\eqref{OHNOstuffle} is already known (although we did not find a suitable reference in the literature for the finite version of the stuffle identity); hence, we
only sketch the simple arguments. By definition
\begin{equation*}
 \zeta_N(\mathbf{k}) = \sum_{  N\ge n_1 > n_2 > \dots >  n_{r}\ge 1}\frac{1}{n_1^{k_1} n_2^{k_2} \dots n_{r}^{k_{r}}}
 =\sum_{n_1=1}^{n}\frac{1}{n_1^{k_1}}\sum_{n_2=1}^{n_1-1}\frac{1}{n_2^{k_2}}\dots \sum_{n_r=1}^{n_{r-1}-1}\frac{1}{n_r^{k_r}}.
\end{equation*}
Consequently,
\begin{equation*}
\zeta_{N}(\mathbf{k})\zeta_N(\mathbf{h})=
\sum_{n_1=1}^{n}\frac{1}{n_1^{k_1}}\sum_{n_2=1}^{n_1-1}\frac{1}{n_2^{k_2}}\dots \sum_{n_r=1}^{n_{r-1}-1}\frac{1}{n_r^{k_r}}\sum_{m_1=1}^{N}\frac{1}{m_1^{h_1}}\zeta_{m_1-1}(h_2,\dots,h_s).
\end{equation*}
Since by definition of $\zeta_{N}(\mathbf{k})$ the summation ranges are given by $N\ge n_1 > n_2 > \dots >  n_{r}\ge 1$, we can simple split the summation range of $m_1$ into the following parts
\begin{equation*}
\sum_{m_1=1}^{N}=\sum_{m_1=1}^{n_{r}-1}+\sum_{m_1=n_{r}}+
\sum_{\ell=1}^{r-1}\biggl(\sum_{m_1=n_{r+1-\ell}+1}^{n_{r-\ell}-1}+\sum_{m_1=n_{r-\ell}}\biggr)+
\sum_{m_1=n_1+1}^{N}.
\end{equation*}
The terms corresponding to single term sums of the form $m_1=n_{\ell}$, $1\le \ell \le r$ are merged into the corresponding sums in $\zeta_{N}(\mathbf{k})$; then we recursively apply the same procedure to products $$\zeta_{m_1-1}(k_{\ell+1},\dots,k_r)\zeta_{m_1-1}(h_2,\dots,h_s),$$
which are of smaller length $r-\ell+s-1<r+s$. Concerning the remaining sums we simply interchange summations with the corresponding sums in $\zeta_{N}(\mathbf{k})\zeta_N(\mathbf{h})$ and repeat the same procedure to the arising products of finite multiple zeta functions, which are also of smaller length. This proves the stated result.
\end{proof}
\begin{remark}
Evidently, as remarked in~\cite{Bor}, the relative order of the two strings is preserved, but elements
of the two strings may also be shoved together into a common slot (stuffing), thereby
reducing the depth. 
\end{remark}

Subsequently we use the notation $\mathbf{\ell}_{n}^{[1]}=(\ell_1,\ell_2-\ell_1,\dots,n-\ell_{h-1}-1)$,
$\mathbf{\ell}_{n}^{[2]}=(\ell_2,,\ell3-\ell_2,\dots,n-\ell_{h-1}-1)$, and $\mathbf{k}=(k_2,\dots,k_r)$.
For the simplification of $S=\zeta_{n_1}^{\ast}(\{1\}_{n-1})\zeta_{n_1-1}(k_2,\dots,k_r)$, as given in~\eqref{OHNOs},
we apply the stuffle identity of Lemma~\ref{OHNOstufflelem} to values $\zeta_{n_1-1}(\mathbf{\ell}_{n}^{[1]})\zeta_{n_1-1}(\mathbf{k})$ and $\zeta_{n_1-1}(\mathbf{\ell}_{n}^{[2]})\zeta_{n_1-1}(\mathbf{k})$. Hence, $\xi_{k_1,\dots,k_r}(n)$ can be completely evaluated into finite sums of multiple zeta functions, which proves the second part of Theorem~\ref{OHNOthe2}.
We state the explicit evaluation of $\xi_{k_1,\dots,k_r}(n)$ in the following corollary.
\begin{coroll}
\label{OHNOthe3}
For $k_1,\dots,k_r,n\in\NN$ the function $\xi_{k_1,\dots,k_r}(n)$ is given by
\begin{equation*}
\begin{split}
\xi_{k_1,\dots,k_r}(n)&=\sum_{h=1}^{n-1}\sum_{1\le \ell_1<\ell_2<\dots<\ell_{h-1}<n-1}
\sum_{\mathbf{f}\in\stuffi(\mathbf{\ell}_{n}^{[1]},\mathbf{k})}\zeta(k_1+1,\mathbf{f})\\
&\quad+ \sum_{h=1}^{n-1}\sum_{1\le \ell_1<\ell_2<\dots<\ell_{h-1}<n-1}
\sum_{\mathbf{f}\in\stuffi(\mathbf{\ell}_{n}^{[2]},\mathbf{k})}\zeta(k_1+1+\ell_1,\mathbf{f}),
\end{split}
\end{equation*}
with respect to the notation $\mathbf{\ell}_{n}^{[1]}=(\ell_1,\ell_2-\ell_1,\dots,n-\ell_{h-1}-1)$,
$\mathbf{\ell}_{n}^{[2]}=(\ell_2,\ell_3-\ell_2,\dots,n-\ell_{h-1}-1)$, and $\mathbf{k}=(k_2,\dots,k_r)$.
\end{coroll}

\end{document}